\documentclass[a4paper]{amsart}

\usepackage{amsmath,amssymb,amsthm}
\usepackage[all]{xy}
\usepackage{graphicx}
\usepackage[colorlinks=true,backref=page]{hyperref}
\usepackage[utf8]{inputenc}

\newtheorem{theorem}{Theorem}[section]
\newtheorem{proposition}[theorem]{Proposition}
\newtheorem{lemma}[theorem]{Lemma}
\newtheorem{corollary}[theorem]{Corollary}

\theoremstyle{definition}
\newtheorem{definition}[theorem]{Definition}

\newtheorem{conjecture}[theorem]{Conjecture}

\newtheorem{question}[theorem]{Question}

\theoremstyle{remark}

\numberwithin{equation}{section}

\newcommand{\Z}{\mathbb{Z}}

\newcommand{\TC}{\mathsf{TC}}
\newcommand{\cat}{\mathsf{cat}}
\newcommand{\secat}{\mathsf{secat}}
\newcommand{\wgt}{\mathsf{wgt}}
\newcommand{\Hom}{\mathrm{Hom}}
\newcommand{\Ext}{\mathrm{Ext}}
\newcommand{\id}{\mathrm{id}}

\title[TC of non-simply connected spaces]{On the topological complexity of non-simply connected spaces}

\author[Y. Minowa]{Yuki Minowa}
\address{Department of Mathematics, Kyoto University, Kyoto, 606-8502, Japan}
\email{minowa.yuki.48z@st.kyoto-u.ac.jp}

\begin{document}

\begin{abstract}
  Topological complexity is a numerical homotopy invariant that measures the instability of motion planning in a space. To study the topological complexity of non-simply connected spaces, Costa and Farber introduced a cohomology class whose nilpotency gives a lower bound of topological complexity. Farber and Mescher constructed a spectral sequence that evaluates this nilpotency without direct computation. We extend these results with respect to a group homomorphism. As an application, we determine the topological complexity of some 3-manifolds with nonabelian fundamental group.
\end{abstract}

\maketitle

\section{Introduction}\label{introduction}

\emph{Topological complexity} is a numerical homotopy invariant introduced by Farber \cite{F} to describe the instability in the motion planning problem. Let $\mathcal{P} X = \{\gamma \colon [0, 1] \to X \}$ denote the path space, and consider the evaluation map
\[
\Pi\colon \mathcal{P}X \to X^2,\quad\gamma \mapsto \left(\gamma(0),\gamma(1) \right).
\]
The topological complexity of a space $X$, denoted by $\TC(X)$, is defined as the minimal integer such that $X^2$ is covered by $k+1$ open sets having local section of $\Pi$. If such an integer does not exist, then we set $\TC(X)=\infty$.

Since its introduction, topological complexity has been studied intensely in several contexts, and has been generalized in several directions. Among them, one of the core interests is a description of the topological complexity of a non-simply connected space, especially $K(\pi, 1)$-space, in terms of its fundamental group. Costa and Farber \cite{CF} introduced a first cohomology class with local coefficients, and provided upper and lower bounds of topological complexity by the nilpotency of the class. This idea can be regarded as an analog of classical theory on the LS-category of $K(\pi, 1)$-space by Berstein \cite{Be}, \v{S}varc \cite{Sv}, Dranishnikov and Rudyak \cite{DR}. However, it is quite hard to apply these bounds into practice, since the coefficients become significantly intricate as the nilpotency increases. There have been some computational results for spaces with abelian fundamental group, see \cite{ACGHV, CV}. 

Farber and Mescher \cite{FM} considered a spectral sequence that evaluates the above nilpotency without direct computation. For a discrete group $G$, they constructed an exact couple
\[
\cdots\xrightarrow{\delta} D_{r,s}\xrightarrow{\varepsilon}E_{r, s}\to D_{r, s+1}\xrightarrow{\delta} D_{r+1, s}\to \cdots
\]
such that the nontriviality of $\mathrm{Ker}(\varepsilon^{(n)})$ in the $n$-th derived couple
\[
    \cdots\xrightarrow{\delta^{(n)}} D^{n}_{r,s}\xrightarrow{\varepsilon^{(n)}}E^{n}_{r-n+1, s+n-1}\to D^{n}_{r-n+1, s+n}\xrightarrow{\delta^{(n)}} D^{n}_{r-n+2, s+n-1}\to \cdots,
\]
implies $\TC(BG)\ge n-1$. They also provided an explicit description of $E_{r, s}$ in terms of group cohomology. As a result, they obtained a strong evaluation of $\TC(BG)$ for a Gromov hyperbolic group $G$. B\l aszczyk et al. \cite{BCE} and Espinosa Baro et al. \cite{EFMO} generalized this spectral sequence to address variants of the motion planning problem. Nevertheless, this method has not been practically utilized for the computation of topological complexity. One possible reason is that the differentials of the spectral sequence were not precisely given, and so it is difficult to depict $E^n$ for $n\ge 2$ in general cases. Another explanation is that their construction is not functorial, which restrains us from applying abundant tools of group cohomology.

In this paper, we propose a new computational method grounded in these developments; we reconstruct their arguments using techniques from homological algebra. We investigate the property of a pair of adjoint functors associated with a group homomorphism, and the derived homomorphisms of $\Ext$ groups. Our main purpose is to extend the above exact couple with respect to a group homomorphism, so as to facilitate the application of group cohomology. By utilizing our method, we determine the topological complexity of certain non-simply connected $3$-manifolds. As in \cite[Theorem 2]{M}, the \emph{general quaternion group} $Q_{8m}$ acts on $S^3$ without fixed point. Then the quotient $S^3/Q_{8m}$ is a $3$-manifold with  fundamental group $Q_{8m}$.
\begin{theorem}
    \label{main}
    For $m\ge 1$, the topological complexity of $S^3/Q_{8m}$ is $6$.
\end{theorem}
\noindent
Remark that Iwase and Miyata \cite{IM} computed $\TC(S^3/Q_8)$ by analysing a cell structure of $S^3/Q_8$.

This paper is structured as follows. Section \ref{top pre} recalls preliminary results on the cohomological estimation of topological complexity. Section \ref{alg pre} prepares some terminology and lemmas on homological algebra. Section \ref{extension} provides an extension of the result of Farber and Mescher. Section \ref{appl} computes $\TC(S^3/Q_{8m})$ by our method and proves Theorem \ref{main}. Finally, Section \ref{rem} lists some remarks and problems that could lead to further investigation.

\subsection*{Notation and conventions}
Throughout this paper, all groups are assumed to be discrete. We denote the integral group ring of a group $G$ by $\Z G$, and the abelian category of left $G$-modules, i.e. $\Z G$-modules, by $G\textbf{-Mod}$. For a $(G, H)$-bimodule $M$ and a left $H$-module $N$, their tensor product over $\Z G$ is denoted by $M\otimes_G N$, regarded as a left $G$-module by the left action on $M$. For left $G$-modules $M$ and $N$, their tensor product over $\Z$ is denoted by $M\otimes N$, equipped with the diagonal $G$-action
\[
g\cdot(m\otimes n):= (g\cdot m)\otimes (g\cdot n)\text{ for }g\in G, m\in M\text{ and }n\in N.
\]

\subsection*{Acknowledgement}
The author is grateful to Daisuke Kishimoto and Toshiyuki Akita for their valuable comments. The author was supported by JST SPRING, Grant Number JPMJSP2110.

\section{Topological preliminaries}\label{top pre}

\subsection{Sectional category and $\TC$-weight}

We first recall the concept of sectional category. The \emph{sectional category} of a map $f\colon X\to Y$, denoted by $\secat(f)$, is the minimal integer $k$ such that $Y$ is covered by $k+1$ open sets over each of which $f$ has a right homotopy inverse. If such an integer does not exist, then we set $\secat(f)=\infty$. Then by definition, there is an identity
\[
\TC(X)=\secat(\Pi\colon \mathcal{P} X \to X^2).
\]
Let $\Delta\colon X\to X^2$ denote the diagonal map. Then since there is a homotopy commutative diagram (cf. \cite[Theorem 7]{F})
\[
  \xymatrix{
     \mathcal{P} X\ar[r]^{\Pi}\ar[d]_\simeq&X^2\ar@{=}[d]\\
    X\ar[r]^{\Delta}&X^2,
  }
\]
we get $\TC(X)=\secat(\Delta\colon X\to X^2)$.

We next recall a cohomological lower bound of topological complexity. Farber and Grant \cite{FG} introduced the $\TC$-weight of a cohomology class, which is an analog of the category weight for LS-category \cite{FH}. Let $X$ be a space with fundamental group $G=\pi_1(X)$, and let $A$ be a $G^2$-module. Then for a space $Y$ and a map $f\colon Y\to X^2$, we obtain a group homomorphism $\pi_1(Y)\to \pi_1(X^2)=G^2$. We denote by $U_f(A)$ the restriction of $A$ by this homomorphism, which is a $\pi_1(Y)$-module; we will provide a precise definition in Section \ref{alg pre 1}.

\begin{definition}(\cite[Definition 1]{FM})
    The $\TC$-\emph{weight} of a nontrivial cohomology class $u\in H^*(X^2; A)$, denoted by $\wgt(u)$, is the maximal integer $k$ such that
    \[
    f^*(u) = 0 \in H^*(Y;U_f(A))\text{ for any map }f\colon Y\to X^2\text{ with }\secat(f^*\Pi)< k.
    \]
\end{definition}
\noindent
We list some essential properties on this invariant. For proofs, see \cite{FG}.

\begin{proposition}
    \label{wgt props}
    Let $X$ be a space as above.
    \begin{enumerate}
        \item For any $u\in H^*(X^2; A)$, $\wgt(u) \le |u|$.
        \item If $\wgt(u) = k$ for some $u\in H^*(X^2; A)$, then $\TC(X)\ge k$.
        \item For $i = 1,\cdots, m$, let $u_i\in H^*(X^2; A_i)$ be cohomology classes with
        \[
        u_1\smile \cdots\smile u_i \neq 0\in H^*(X^2; A_1\otimes \cdots\otimes A_i).
        \]
        Then
        \[
        \wgt(u_1\smile \cdots\smile u_i) \ge \sum_{i=1}^{m}\wgt(u_i).
        \]
        \item Let $u\neq 0\in H^*(X^2; A)$ be a zero-divisor, i.e. a cohomology class satisfying $\Delta^*(u) = 0\in H^*(X; U_\Delta A)$. Then $\wgt(u)\ge 1$.
    \end{enumerate}
\end{proposition}

\subsection{Canonical class and spectral sequence}

In this section, we recall some previous results regarding the first cohomology class of Costa and Farber \cite{CF}. Let $X$ be a space with fundamental group $G= \pi_1(X)$. We denote by $\Z[G]$ a free $\Z$-module with basis $G$, and equip it with $G^2$-action
\[
(a, b)\cdot c:= acb^{-1}\quad\text{ for }a, b, c\in G.
\]
Let $I_G$ denote the kernel of the augmentation $\varepsilon\colon \Z[G]\to Z$. Remark that $I_G$ is free as a $\Z$-module with basis $\{g-1 \mid g\in G{\setminus\{1\}} \}$. 
\begin{definition}[\text{\cite[\S 2]{CF}}]
     The \emph{canonical class} $\mathfrak{v}_X\in H^1(X^2; I_G)$ is a cohomology class determined by a map $G^2 \to I_G, (a, b)\mapsto ab^{-1}-1$.
\end{definition}
\noindent
This class is canonical in the sense that, for the canonical map $p\colon X\to BG$ associated with the universal covering, we have $(p^2)^{\ast}(\mathfrak{v}_{BG}) = \mathfrak{v}_{X}$. For this reason, we abbreviate $\mathfrak{v}_{BG}$ to $\mathfrak{v}$. Remark that, if $X$ is simply connected, then $I_G = 0$ and so $\mathfrak{v}_X=0$. One may also check that $\mathfrak{v}_X$ is a zero-divisor; therefore if $\mathfrak{v}_X^k\neq 0\in H^k(X^2; I_G^{\otimes k})$, then $\TC(X)\ge k$. In fact, they obtained a shaper estimation.

\begin{theorem}[\text{\cite[Theorem 7]{CF}}]
    Let $X$ be a CW-complex of dimension $n\ge 2$. Then $\TC(X)<2n$ if and only if $\mathfrak{v}_X^{2n}= 0 \in H^{2n}(X^2; I_G^{\otimes 2n})$.
\end{theorem}

We next consider a spectral sequence introduced by Farber and Mescher \cite{FM}, with slight modifications of notation. Let $M$ be a $G^2$-module that is $\Z$-free. Then by the flatness of $M$, there is a short exact sequence of $G^2$-modules
\[
0\to I_G\otimes M \to \Z[G]\otimes M \xrightarrow{\varepsilon} M  \to 0,
\]
where $\varepsilon$ denotes the augmentation; notice that $I_G\otimes M$ is also $\Z$-free. By iterating this process, we obtain a short exact sequence of $G^2$-modules
\[
0\to I_G^{\otimes s+1} \to \Z[G]\otimes I_G^{\otimes s} \xrightarrow{\varepsilon} I_G^{\otimes s} \to 0.
\]
Then for a $G^2$-module $A$, there is a long exact sequence of $\Z$-modules
\begin{align}
    \label{Ext LES}
    \cdots\xrightarrow{\delta}\Ext^r_{G^2}(I_G^{\otimes s}, A)\xrightarrow{\varepsilon}\Ext^r_{G^2}&(\Z[G]\otimes I_G^{\otimes s}, A)\\
    \notag&\to \Ext^r_{G^2}(I_G^{\otimes s+1}, A)\xrightarrow{\delta}\Ext^{r+1}_{G^2}(I_G^{\otimes s}, A)\xrightarrow{\varepsilon}\cdots,
\end{align}
where $\delta$ denotes the Bockstein homomorphism. We set
\[
D^1_{r, s} := \Ext^r_{G^2}(I_G^{\otimes s}, A), E^1_{r, s} := \Ext^r_{G^2}(\Z[G]\otimes I_G^{\otimes s}, A), \delta^{(1)}= \delta\text{ and }\varepsilon^{(1)} = \varepsilon
\]
and construct the $n$-th derived couple (cf. \cite[VIII. 1.]{HS})
\begin{equation}
    \label{der seq}
    \cdots\xrightarrow{\delta^{(n)}} D^{n}_{r,s}\xrightarrow{\varepsilon^{(n)}}E^{n}_{r-n+1, s+n-1}\to D^{n}_{r-n+1, s+n}\xrightarrow{\delta^{(n)}} D^{n}_{r-n+2, s+n-1}\to \cdots
\end{equation}
\noindent
so that $E^n_{\ast, \ast}$ is a homology spectral sequence with a differential $d^n$ of bidegree $(1-n, n)$. Remark that $D^1_{r,0} = \Ext^r_{G^2}(\Z, A) = H^r(G^2; A)$, and that
\[
D^n_{r,0} = \mathrm{Im}(\delta\circ\cdots\circ\delta\colon D^1_{r-n+1, n-1}\to D^1_{r,0}).
\]
We now consider a function
\[
\mathbf{ev}\colon  I_G^{\otimes n} \otimes \Hom_\Z(I_G^{\otimes n}, A),\quad \mathbf{ev}(x_1\otimes \cdots\otimes x_n\otimes f):= f(x_n\otimes \cdots\otimes x_1)
\]
and recall:
\begin{theorem}[\text{\cite[Corollary 7.4]{FM}}]
    \label{FM Thm}
    For a cohomology class $z\in H^r(G^2; A) = D^1_{r,0}$ and $n\ge 1$, the following are equivalent:
    \begin{enumerate}
        \item $z\in D^{n+1}_{r,0}$.
        \item $\varepsilon^{(i)}(z)=0\in E^{i}_{r-i+1, s+i-1}$ for $i =1, \cdots, n$.
        \item $z = \mathbf{ev}_*(\mathfrak{v}^{n}\smile u)$ for some $u\in H^{r-n}(G^2; \Hom_\Z(I_G^{\otimes n}, A))$.
    \end{enumerate}
\end{theorem}
\noindent
Notice that the equivalence of $(1)$ and $(3)$ holds for $n=0$. For $n=1$, cohomology classes satisfying above conditions are characterized as follows:
\begin{proposition}[\text{\cite[Corollary 7.6]{FM}}]
    \label{FM zero}
    For a cohomology class $z\in H^r(G^2; A)$, the condition $z\in D^{2}_{r,0}$ is equivalent to $z$ being a zero-divisor.
\end{proposition}
\noindent
We also state a supplementary lemma.
\begin{lemma}
    \label{FM supp}
    Let $A$ be an associative $\Z G^2$-algebra with multiplication $\mu\colon A\otimes A\to A$, and suppose that $z_1\in D^{n_1+1}_{{r_1},0}$ and $z_2\in D^{n_2+1}_{{r_2},0}$. Then $z_1\smile z_2\in D^{n_1+n_2+1}_{{r_1+r_2},0}$.
\end{lemma}

\begin{proof}
    Define a map $\lambda\colon \Hom_\Z(I_G^{\otimes n_1}, A)\otimes \Hom_\Z(I_G^{\otimes n_2}, A) \to \Hom_\Z(I_G^{\otimes n_1+ n_2}, A)$ by
    \begin{align*}
        \lambda(f_1\otimes f_2)&(x_1\otimes \cdots\otimes x_{n_1}\otimes\cdots\otimes x_{n_1+n_2})\\
        &= \mu(f_1(x_{n_1}\otimes\cdots\otimes x_1)\otimes f_2(x_{n_1+n_2}\otimes\cdots\otimes x_{n_1+1})).
    \end{align*}
    \noindent
    It is straightforward to check that $\lambda$ is a $G^2$-homomorphism, and that there is a commutative diagram
    \[
    \xymatrix{
        I_G^{\otimes n_1}\otimes \Hom_\Z(I_G^{\otimes n_1}, A)\otimes I_G^{\otimes n_2}\otimes \Hom_\Z(I_G^{\otimes n_2}, A)\ar^-{\mu\circ(\mathbf{ev}\otimes\mathbf{ev})}[r]\ar[d] 
        & A\\
        I_G^{\otimes n_1}\otimes I_G^{\otimes n_2}\otimes \Hom_\Z(I_G^{\otimes n_1}, A)\otimes \Hom_\Z(I_G^{\otimes n_2}, A)\ar_-{\id\otimes \lambda}[r] & I_G^{\otimes n_1+ n_2}\otimes \Hom_\Z(I_G^{\otimes n_1+ n_2}, A),
        \ar_{\mathbf{ev}}[u]
    }
    \]
    where the left column is the switch map of the second and third component. Thus by the definition of cup product and Theorem \ref{FM Thm}, the statement is proved.
\end{proof}

\section{Algebraic preliminaries}\label{alg pre}

In this section, we collect the necessary preliminaries from homological algebra for the subsequent discussion.

\subsection{The induction and restriction functor}\label{alg pre 1}

We consider two functors associated with a group homomorphism $\alpha\colon H\to G$; the \emph{induction} $F_\alpha\colon H\textbf{-Mod}\to G\textbf{-Mod}$ and the \emph{restriction} $U_\alpha\colon G\textbf{-Mod}\to H\textbf{-Mod}$. To define the induction, we first regard $\Z G$ as a $(G, H)$-bimodule via the left translation and the right $H$-action
\[
a\cdot b = a \alpha(b)\quad\text{ for }a\in G \text{ and }b \in H.
\]
Then for a left $H$-module $M$, we obtain a left $G$-module $F_\alpha M := \Z G\otimes_H M$. We also equip a left $G$-module $N$ with an $H$-action
\[
a\cdot n := \alpha(a)\cdot n\quad\text{ for }a\in H \text{ and } n\in N,
\]
and denote it by $U_\alpha N$. It is easily verified that these correspondences are functorial, and that $F_\alpha$ is a left adjoint of $U_\alpha$. We denote the natural isomorphism by
\[
\varphi_\alpha\colon \Hom_G(F_\alpha M, N)\to \Hom_H(M, U_\alpha N),
\]
the unit $\varphi_\alpha(\id)\colon M\to U_\alpha F_\alpha M$ by $\eta_\alpha$ and the counit $\varphi_\alpha^{-1}(\id)\colon F_\alpha U_\alpha N\to N$ by $\varepsilon_\alpha$.

We now consider a commutative diagram of groups
\begin{equation}
    \label{group square}
\xymatrix{
H'\ar^{\beta}[r]\ar_{\kappa}[d] & G'\ar^{\iota}[d]\\
H\ar_{\alpha}[r] & G.
}
\end{equation}
\noindent
Remark that $F_\iota F_\beta = F_\alpha F_\kappa$ and $U_\beta U_\iota = U_\kappa U_\alpha$. Then for an $H'$-module $L$, $G'$-module $M$ and $f\in \Hom_{H'}(L, U_\beta M)$, we obtain a $H'$-homomorphism
\[
\xymatrix{
    L \ar^{f}[rr] & & U_\beta M \ar^-{U_\beta \eta_\iota}[rr] & & U_\beta U_\iota F_\iota M = U_\kappa U_\alpha F_\iota M.
}
\]
For simplicity, we denote by $\overline{f}$ a homomorphism $\varphi_\kappa^{-1}(U_\beta \eta_\iota \circ f)\colon F_\kappa L \to U_\alpha F_\iota M$.
Regarding this construction, we list some lemmas for future reference.
\begin{lemma}
    \label{f bar}
    Let $f$ as above, and let $\id\colon U_\beta M\to U_\beta M$ denote the identity. Then:
    \begin{enumerate}
        \item $\varepsilon_\alpha \circ F_\alpha(\overline{\id}) = F_\iota(\varepsilon_\beta)\colon F_\iota F_\beta U_\beta M\to F_\iota M$.
        \item $\overline{f} = \overline{\id}\circ F_\kappa(f) \colon F_\kappa L \to U_\alpha F_\iota M$.
    \end{enumerate}
\end{lemma}

\begin{proof}
    \begin{enumerate}
    \item By the naturality of $\varphi^{-1}_\alpha$, we obtain
    \begin{align*}
        \varepsilon_\alpha \circ F_\alpha(\overline{\id}) &= \varphi^{-1}_\alpha(\overline{\id})\\
        &= \varphi^{-1}_\alpha (\varphi_\kappa^{-1}(U_\beta \eta_\iota))\\
        &= \varphi^{-1}_\iota (\varphi_\beta^{-1}(U_\beta \eta_\iota))\\
        &= \varphi^{-1}_\iota (\eta_\iota\circ\varepsilon_\beta) = F_\iota(\varepsilon_\beta).
    \end{align*}
    Thus since $\varphi^{-1}_\alpha$ is an isomorphism, the statement is proved.
    \item We also have
    \[\varphi_\kappa(\overline{\id}\circ F_\kappa(f)) = \varphi_\kappa (\overline{\id})\circ f.
    \]
    The statement is proved in the same way as above.
    \end{enumerate}
\end{proof}

\begin{lemma}
    \label{f square}
    Let $L'$ be an $H'$-module, $M'$ a $G$-module and $h\in \Hom_{H}(L', U_\alpha M')$. Then $U_\kappa h\in \Hom_{H}(U_\kappa L', U_\beta U_\iota M')$ and there is a commutative diagram
    \[
    \xymatrix{
        F_\kappa U_\kappa L' \ar^{\varepsilon_\kappa}[rr]\ar_{\overline{U_\kappa h}}[d]& & L'\ar^{h}[d]\\
        U_\alpha F_\iota U_\iota M' \ar_{U_\alpha\varepsilon_\iota}[rr]& & U_\alpha M'.
    }
    \]
\end{lemma}

\begin{proof}
    We have
    \begin{align*}
        \varphi_\kappa(U_\alpha \varepsilon_\iota\circ \overline{U_\kappa h})&= U_\kappa U_\alpha \varepsilon_\iota\circ (U_\beta \eta_\iota \circ U_\kappa h)\\
        &= U_\beta U_\iota \varepsilon_\iota \circ U_\beta \eta_\iota \circ U_\kappa h \\
        &= U_\beta (U_\iota \varepsilon_\iota \circ \eta_\iota) \circ U_\kappa h = U_\kappa h = \varphi_\kappa (h\circ \varepsilon_\kappa),
    \end{align*}
    where the fourth identity follows from the counit-unit equations. Thus since $\varphi_\kappa$ is an isomorphism, the proof is finished.
\end{proof}

\subsection{Ext group homomorphisms} Let $\alpha\colon H\to G$ be a group homomorphism. Then, as in \cite[IV. 12.]{HS}, the natural isomorphism 
\[
\varphi_\alpha\colon \Hom_G(F_\alpha M, N)\to \Hom_H(M, U_\alpha N)
\]
induces a natural homomorphism of Ext groups
\begin{equation}
    \label{Ext Phi}
    \Phi_\alpha\colon \Ext^r_G(F_\alpha M, N)\to \Ext^r_H(M, U_\alpha N) \quad\text{for}\quad r\ge 0.
\end{equation}
Remark that, if $\alpha$ is an inclusion, then $\Phi_\alpha$ is a natural isomorphism (cf. \cite[VI. Corollary 1.4.]{HS}). We also set a natural homomorphism
\[
\Theta_\alpha\colon \Ext^r_G(M, N)\xrightarrow{\varepsilon_\alpha^*} \Ext^r_G(F_\alpha U_\alpha M, N) \xrightarrow{\Phi_\alpha} \Hom_H(U_\alpha M, U_\alpha N).
\]
We now consider the commutative diagram \eqref{group square}.
\begin{proposition}
    \label{Ext square}
    Let $A$ be a $G$-module with trivial action, and let $f\in \Hom_{H'}(L, U_\beta M)$. Then there is a commutative diagram of abelian groups
\[
\xymatrix{
        \Ext^r_G(F_\iota M, A)  \ar^{\overline{f}^*\circ\Theta_\alpha}[rr]\ar_{\Phi_\iota}[d]& & \Ext^r_H(F_\kappa L, A)\ar^{\Phi_\kappa}[d]\\
        \Ext^r_{G'}(M, A) \ar_{f^*\circ\Theta_\beta}[rr] & & \Ext^r_{H'}(L, A),
    }
\]
where $A$ is regarded as a trivial $G', H$ and $H'$-module.
\end{proposition}

\begin{proof}
    Let $\id\colon U_\beta M\to U_\beta M$ denote the identity. We first consider the diagram
    \[
    \xymatrix{
    \Ext^r_G(F_\iota M, A) \ar^-{\varepsilon_\alpha^*}[r]\ar@{=}[d] &\Ext^r_G(F_\alpha U_\alpha F_\iota M, A) \ar^{\Phi_\alpha}[r]\ar^{{F_\alpha(\overline{\id})}^*}[d] & \Ext^r_H(U_\alpha F_\iota M, A) \ar^{\overline{f}^*}[r]\ar^{\overline{\id}^*}[d] & \Ext^r_H(F_\kappa L, A)\ar@{=}[d]\\
    \Ext^r_G(F_\iota M, A) \ar_-{F_\iota(\varepsilon_\beta)^*}[r] &\Ext^r_G(F_\alpha F_\kappa U_\beta M, A) \ar_{\Phi_\alpha}[r] & \Ext^r_H(F_\kappa U_\beta M, A) \ar_{F_\kappa(f)^*}[r] & \Ext^r_H(F_\kappa L, A).
    }
    \]
    Then by Lemma \ref{f bar} and the naturality of $\Phi_\alpha$, all squares are commutative. 
    
    We next show that the following diagram is commutative;
    \[
    \xymatrix{
    \Ext^r_G(F_\iota M, A) \ar^-{F_\iota(\varepsilon_\beta)^*}[r] \ar[d]_{\Phi_\iota}&\Ext^r_G(F_\iota F_\beta U_\beta M, A) \ar^{\Phi_\alpha}[r] \ar_{\Phi_\iota}[d]& \Ext^r_H(F_\kappa U_\beta M, A) \ar^{F_\kappa(f)^*}[r] \ar_{\Phi_\kappa}[d]& \Ext^r_H(F_\kappa L, A)\ar^{\Phi_\kappa}[d]\\
    \Ext^r_{G'}(M, A) \ar_-{{\varepsilon_\beta}^*}[r] & \Ext^r_{G'}(F_\beta U_\beta M, A) \ar_{\Phi_\beta}[r] & \Ext^r_{H'}(U_\beta M, A) \ar_{f^*}[r] & \Ext^r_{H'}(L, A).
    }
    \]
   The commutativity of the left and right squares follows from the naturality of $\Phi_\iota$ and $\Phi_\kappa$. The construction of the Ext homomorphisms \eqref{Ext Phi} is obviously functorial, and so the middle square is commutative. Thus by combining these diagrams, the statement is proved.
\end{proof}

\subsection{Representations}

For a group homomorphism $\alpha\colon H\to G$, we recall that the natural isomorphism $\varphi_\alpha\colon \Hom_G(F_\alpha M, N)\to \Hom_H(M, U_\alpha N)$ is given by
\begin{equation}
    \label{phi rep}
    \varphi_\alpha(g)(m)= g(1\underset{H}{\otimes} m),\quad \varphi_\alpha^{-1}(h)(a\underset{H}{\otimes} m)= a\cdot h(m),
\end{equation}
for $g\in \Hom_G(F_\alpha M, N)$, $h\in \Hom_H(M, U_\alpha N)$,  $a\in G$ and $m\in M$. In particular, if $\iota\colon G' \hookrightarrow G$ is an inclusion, then there is an isomorphism (cf. \cite[III. 5.]{Br})
\begin{equation}
    \label{ind res}
F_\iota U_\iota N = \Z G\otimes_{G'} U_\iota N \cong \Z[G/G'] \otimes N,\quad a\underset{G'}{\otimes} n\mapsto [a]\otimes (a\cdot n)
\end{equation}
\noindent
for $a\in \Z G$ and $n\in N$, where $[a]$ denotes the representative of $a$ in $\Z[G/G']$ and $G$ acts diagonally on the RHS. Thus by \eqref{phi rep}, the counit $\varepsilon_\iota= \varphi_\alpha^{-1}(\id)\colon F_\iota U_\iota N\to N$ coincides with the augmentation
\[
\varepsilon\colon\Z[G/G'] \otimes N \to N,\quad [a]\otimes n\mapsto n.
\]

We now resume considering the commutative diagram of groups
\begin{equation*}
\xymatrix{
H'\ar^{\beta}[r]\ar_{\kappa}[d] & G'\ar^{\iota}[d]\\
H\ar_{\alpha}[r] & G,
}
\end{equation*}
\noindent
with an assumption that $\kappa$ and $\iota$ are inclusions.

\begin{proposition}
    \label{bar rep}
    For $f\in \Hom_{H'}(L, U_\beta M)$, the homomorphism $\overline{f}\colon F_\kappa L\to U_\alpha F_\iota M$ is represented by
    \[
    a\underset{H'}{\otimes} m \mapsto \alpha(a)\underset{G'}{\otimes}f(m)\quad\text{ for }a\in \Z H\text{ and }m\in L.
    \] 
\end{proposition}

\begin{proof}
    By \eqref{phi rep}, the unit $\eta_\iota=\varphi_\iota(\id)\colon M\to U_\iota F_\iota M$ is written as 
    \[
    m\mapsto \id(1\underset{G'}{\otimes}m) = 1\underset{G'}{\otimes}m \quad\text{ for }m\in M,
    \]
    and so we have
    \[
    L\xrightarrow{f}U_\beta L\xrightarrow{U_\beta \eta_\iota}U_\beta U_\iota F_\iota M\colon m\mapsto 1\underset{G'}{\otimes}f(m) \text{ for }m\in L.
    \]
    Remark that $U_\alpha F_\iota M$ is equipped with an $H$-action
    \[
    a\cdot(b\underset{G'}{\otimes}m) = (\alpha(a)\cdot b)\underset{G'}{\otimes}m \quad\text{ for }a\in H,\, b \in \Z G\text{ and }m\in M.
    \]
    Thus since $\overline{f}=\varphi_\kappa^{-1}(U_\beta \eta_\iota \circ f)$, we obtain the description by \eqref{phi rep}, as stated.
\end{proof}

\begin{corollary}
    \label{bar rep 2}
    For $h\in \Hom_{H}(L', U_\alpha M')$, the homomorphism 
    \[
    \overline{U_\kappa h}\colon F_\kappa U_\kappa L' = \Z[H/H']\otimes L' \to U_\alpha(\Z[G/G']) \otimes U_\alpha M' = U_\alpha F_\iota U_\iota M'
    \]
    in Lemma \ref{f square} is represented by
    \[
       [a]\otimes m\mapsto [\alpha(a)]\otimes h(m) \quad\text{ for }a\in \Z H\text{ and }m\in L'.
    \]
\end{corollary}

\begin{proof}
    Since $G$ acts diagonally on $\Z[G/G']\otimes M'$, we have $U_\alpha(\Z[G/G']) \otimes U_\alpha M' = U_\alpha F_\iota U_\iota M'$. We now consider a commutative diagram
    \[
    \xymatrix{
        \Z[H/H']\otimes L' \ar[r]\ar_{\cong}[d] & U_\alpha(\Z[G/G']) \otimes U_\alpha M'\\
        \Z H \underset{H'}{\otimes} U_\kappa L'\ar_{\overline{U_\kappa h}}[r] & U_\alpha(\Z G \underset{G'}{\otimes} U_\iota M')\ar_{\cong}[u],
    }
    \]
    where the vertical isomorphisms are in \eqref{ind res}. Thus by \eqref{ind res} and Proposition \ref{bar rep}, the top row is given by
    \begin{align*}
        [a]\otimes m &\mapsto a\underset{H'}{\otimes} (a^{-1}\cdot m)\\
        &\mapsto \alpha(a)\underset{G'}{\otimes} (U_\kappa h(a^{-1}\cdot m)) = \alpha(a)\underset{G'}{\otimes} (\alpha(a))^{-1}\cdot h(m)\\ 
        &\mapsto [\alpha(a)]\otimes h(m)
    \end{align*}
    as stated.
\end{proof}

\section{Reformulation and extension}\label{extension}

In this section, we restate the result of Farber and Mescher \cite{FM} in terms of homological algebra and group cohomology.

\subsection{Diagonal inclusion}

For a group $G$, let $\Delta\colon G\hookrightarrow G^2$ denote the diagonal inclusion. Then the quotient $G^2$-module $\Z[G^2/\Delta(G)]$ is identified with $\Z[G]$ by
\[
\Z[G^2/\Delta(G)]\ni [(a, b)]\mapsto ab^{-1}\in \Z[G] \quad\text{ for }a, b\in G.
\]
We now take a group homomorphism $\alpha\colon H\to G$, set $\alpha^2= (\alpha, \alpha)\colon H^2\to G^2$ and consider a commutative square
\[
\xymatrix{
    H \ar^{\alpha}[r]\ar_{\Delta_H}[d]& G\ar^{\Delta_G}[d]\\
    H^2 \ar_{\alpha^2}[r] & G^2.
}
\]
We apply Lemma \ref{f bar} to an $H^2$-homomorphism $h\colon M'\to U_\alpha^2 M$, obtaining a commutative diagram
\[
\xymatrix{
    0 \ar[r] & I_H\otimes L'\ar[r]\ar[d] & \Z[H]\otimes L'\ar[r]\ar_{\overline{U_{\Delta_H}h}}[d] & L' \ar[r]\ar^{h}[d] & 0\\
    0 \ar[r] & U_{\alpha^2}(I_G)\otimes U_{\alpha^2}M'\ar[r] & U_{\alpha^2}(\Z[G])\otimes U_{\alpha^2}M'\ar[r] & U_{\alpha^2}M' \ar[r] & 0,
}
\]
By Corollary \ref{bar rep 2}, we have $\overline{U_{\Delta_H}h}(b\otimes m) =\alpha(b)\otimes h(m)$ for $b\in H$ and $m\in L'$. Iterating this construction on $h_0 = \id\colon \Z\to \Z$, a commutative diagram
\begin{equation}
    \label{aug nat}
    \xymatrix{
    0 \ar[r] & (I_H)^{\otimes s+1}\ar[r]\ar_{h_{s+1}}[d] & \Z[H]\otimes (I_H)^{\otimes s}\ar^-{\varepsilon}[r]\ar_{\overline{U_{\Delta_H}h_s}}[d] & (I_H)^{\otimes s} \ar[r]\ar^{h_s}[d] & 0\\
    0 \ar[r] & (U_{\alpha^2}(I_G))^{\otimes s+1}\ar[r] & U_{\alpha^2}(\Z[G])\otimes (U_{\alpha^2}(I_G))^{\otimes s}\ar^-{\varepsilon}[r] & (U_{\alpha^2}(I_G))^{\otimes s} \ar[r] & 0,
}
\end{equation}
\noindent
is obtained, where $H^2$ acts diagonally on each module. Observe that
\begin{equation}
    \label{h rep}
    h_{s}((b_1-1)\otimes\cdots\otimes(b_s-1)) = ((\alpha(b_1)-1)\otimes\cdots\otimes(\alpha(b_s)-1))
\end{equation}
\noindent
for $b_i\in H\setminus\{1\}$.

\begin{lemma}
    \label{Ext exact hom}
    Let $A$ be a $G^2$-module with trivial action, also regarded as a trivial $H^2$-module. Then there is a homomorphism between two long exact sequences
    \[
    \xymatrix{
      \ar^-{\delta}[r]& \Ext^r_{G^2}(I_G^{\otimes s}, A)  \ar^-{\varepsilon}[r]\ar_{{h_{s+1}}^*\circ \Theta_{\alpha^2}}[d]& \Ext^r_{G^2}(\Z[G]\otimes I_G^{\otimes s}, A)  \ar[r]\ar_{{\overline{U_{\Delta_H}h_s}}^*\circ \Theta_{\alpha^2}}[d]& \Ext^r_{G^2}(I_G^{\otimes s+1}, A)  \ar^-{\delta}[r]\ar^{{h_{s}}^*\circ \Theta_{\alpha^2}}[d]& \\
      \ar_-{\delta}[r]& \Ext^r_{H^2}(I_H^{\otimes s}, A)  \ar_-{\varepsilon}[r]& \Ext^r_{H^2}(\Z[H]\otimes I_H^{\otimes s}, A)  \ar[r]& \Ext^r_{H^2}(I_H^{\otimes s+1}, A)  \ar_-{\delta}[r]& ,
    }
    \]
    where $\delta$ denotes the connecting homomorphism of $\Ext$ long exact sequence.
\end{lemma}

\begin{proof}
    As in \cite[IV. 12.]{HS}, the homomorphism $\Theta_{\alpha^2}$ is natural and commutes with the connecting homomorphism. Thus the statement is proved.
\end{proof}
\noindent
The middle terms $\Ext^r_{G^2}(\Z[G]\otimes I_G^{\otimes s}, A)$ and $\Ext^r_{H^2}(\Z[H]\otimes I_H^{\otimes s}, A)$ are of special interest, as there is a natural isomorphism
\begin{equation}
    \label{diag iso}
    \Phi_{\Delta_G}\colon \Ext^r_{G^2}(\Z[G]\otimes I_G^{\otimes s}, A) \cong \Ext^r_{G}((U_{\Delta_G}I_G)^{\otimes s}, A),
\end{equation}
\noindent
which is identical to \cite[Lemma 5.4]{FM}. By applying Proposition \ref{Ext square} to $U_{\Delta_H}h_s$, we get the naturality of the above isomorphism:

\begin{proposition}
    \label{diag nat}
    For a trivial $G^2$-module $A$, there is a commutative diagram
    \[
\xymatrix{
        \Ext^r_{G^2}(\Z[G]\otimes I_G^{\otimes s}, A)  \ar^{\overline{U_{\Delta_H}h_s}^*\circ\Theta_{{\alpha}^2}}[rr]\ar_{\Phi_{\Delta_G}}[d]& & \Ext^r_{H^2}(\Z[H]\otimes I_H^{\otimes s}, A)\ar^{\Phi_{\Delta_H}}[d]\\
        \Ext^r_{G}((U_{\Delta_G}I_G)^{\otimes s}, A) \ar_{({U_{\Delta_H}h_s})^*\circ\Theta_\alpha}[rr] & & \Ext^r_{H}((U_{\Delta_H}I_H)^{\otimes s}, A),
    }
\]
where $A$ is regarded as a trivial $G, H$ and $H^2$-module.
\end{proposition}

\subsection{Centralizer inclucion}

We denote by $\mathcal{O}_G$ the orbit set of $(G\setminus\{1\})^s$ under the conjugational action of $G$, that is,
\[
g\cdot(a_1, \cdots, a_s)=(ga_1g^{-1},\cdots,ga_sg^{-1})
\]
for $g\in G$ and $a_i\in G\setminus\{1\}$. Then, as in the proof of \cite[Theorem 8.1]{FM}, there is an isomorphism of $G$-modules
\[
(U_{\Delta_G}I_G)^{\otimes s} \cong \bigoplus_{O\in \mathcal{O}_G} \Z[O],\quad (a_1-1)\otimes\cdots\otimes(a_s-1) \mapsto (a_1, \cdots, a_s).
\]
For a tuple $a= (a_1, \cdots, a_s) \in (G\setminus\{1\})^s$, let $O_a$ denote the orbit containing $a$, $C_G(a)$ the centralizer of $a$, and $\iota_a\colon C_G(a)\hookrightarrow G$ an inclusion. Remark that $C_G(a)$ is the intersection of the centralizers of the elements $a_1, \cdots, a_s$ under the conjugational action. Then by the orbit-stabilizer theorem, we have
\begin{equation}
    \label{orb iso}
    \Z[O_a] \cong \Z[G/C_a(G)] \cong F_{\iota_a} U_{\iota_a} \Z, \quad g\cdot(a_1, \cdots, a_s) \mapsto [g],
\end{equation}
\noindent
where the $G$-action on $\Z$ is trivial. Thus we obtain
\begin{align}
    \label{cent iso}
    \Ext^r_{G}((U_{\Delta_G}I_G)^{\otimes s}, A) &\cong \prod_{O_a\in \mathcal{O}_G} \Ext^r_{G}(\Z[O_a], A)\\
    \notag&\cong \prod_{O_a\in \mathcal{O}_G} \Ext^r_{C_G(a)}(\Z, A) = \prod_{O_a\in \mathcal{O}_G} H^r(C_G(a); A)
\end{align}
\noindent
as stated in \cite[Theorem 8.1]{FM}. For $a \in (G\setminus\{1\})^s$, we denote the inclusion and projection of the corresponding summand by $j_a$ and $q_a$, respectively.

We now resume considering a group homomorphism $\alpha\colon H\to G$. Taking tuples $a\in (G\setminus\{1\})^s$ and $b\in (H\setminus\{1\})^s$, we aim to compute the composite
\begin{align*}
    \Psi_{a, b}\colon H^r(C_G(a); A) &\xrightarrow{j_a} \Ext^r_{G}((U_{\Delta_G}I_G)^{\otimes s}, A) \\
    &\xrightarrow{({U_{\Delta_H}h_s})^*\circ\Theta_\alpha} \Ext^r_{H}(U_{\Delta_H}I_H)^{\otimes s}\xrightarrow{q_b} H^r(C_H(b); A),
\end{align*}
where the middle homomorphism is the same as the bottom in the diagram of Theorem \ref{diag nat}. We recall \eqref{h rep} and observe that
\[
\psi_{a, b}\colon\Z[O_b] \hookrightarrow (U_{\Delta_H}I_H)^{\otimes s} \xrightarrow{{U_{\Delta_H}h_s}} U_\alpha (U_{\Delta_G}I_G)^{\otimes s} \twoheadrightarrow U_\alpha\Z[O_a]
\]
is described as $(b_1,\cdots, b_s)\mapsto (\alpha(b_1),\cdots, \alpha(b_s))$. Then, if $\alpha(b)\notin \mathcal{O}_a$, the homomorphism $\psi_{a, b}$ is trivial. Let us consider the case $\alpha(b)\in \mathcal{O}_a$, where we may assume that $\alpha(b) = a$. In this case, there is a commutative diagram
\[
\xymatrix{
    C_H(b) \ar^{\alpha}[r]\ar_{\iota_b}[d]& C_G(a) \ar^{\iota_a}[d]\\
    H \ar_{\alpha}[r] & G,
}
\]
where the top row denotes the restriction.
\begin{proposition}
    \label{cent nat}
    For a trivial $G$-module $A$, there is a commutative diagram
    \[
    \xymatrix{
        \Ext^r_G(\Z[O_a], A)  \ar^{\psi_{a, b}^*\circ\Theta_\alpha}[rr]\ar_{\Phi_{\iota_a}}[d]& & \Ext^r_H(\Z[O_b], A)\ar^{\Phi_{\iota_b}}[d]\\
        H^r(C_G(a); A) \ar_{\alpha^*}[rr] & & H^r(C_H(b); A),
    }
\]
\end{proposition}

\begin{proof}
    By Proposition \ref{bar rep 2} and \eqref{orb iso}, we have $\psi_{a, b} = \overline{U_{\iota_b} \id}\colon F_{\iota_b} U_{\iota_b} \Z\to U_\alpha F_{\iota_a} U_{\iota_a} \Z$. Then we apply Proposition \ref{Ext square} to get a commutative diagram
    \[
    \xymatrix{
        \Ext^r_G(\Z[O_a], A)  \ar^{\psi_{a, b}^*\circ\Theta_\alpha}[rr]\ar_{\Phi_{\iota_a}}[d]& & \Ext^r_H(\Z[O_b], A)\ar^{\Phi_{\iota_b}}[d]\\
        \Ext^r_{C_G(a)}(\Z, A) \ar_{\id^*\circ\Theta_\alpha}[rr] & & \Ext^r_{C_H(b)}(\Z, A).
    }
\]
Thus since the bottom row is the induced homomorphism $\alpha^*\colon H^r(C_G(a); A)\to H^r(C_H(b); A)$ (cf. \cite[VI. (2.11.)]{HS}), the proof is finished.
\end{proof}

We now combine \eqref{diag iso}, \eqref{cent iso} and Propositions \ref{diag nat}, \ref{cent nat} to:

\begin{theorem}
    \label{decomp}
    Let $\alpha\colon H\to G$ be a group homomorphism, and let $a\in (G\setminus\{1\})^s$ and $b\in (H\setminus\{1\})^s$ be tuples. If $\alpha(b)\notin \mathcal{O}_a$, then the homomorphism
    \begin{align*}
        \Psi_{a, b}\colon H^r(C_G(a); A) &\xrightarrow{j_a} \Ext^r_{G^2}(\Z[G]\otimes(I_G)^{\otimes s}, A) \\
    &\xrightarrow{(\overline{U_{\Delta_H}h_s})^*\circ\Theta_\alpha} \Ext^r_{H^2}(\Z[H]\otimes(I_H)^{\otimes s}, A)\xrightarrow{q_b} H^r(C_H(b); A)
    \end{align*}
    \noindent
    is trivial. If $\alpha(b)=a$, then $\Psi_{a, b}$ coincides with the induced homomorphism $\alpha^*$.
\end{theorem}

\section{On the topological complexity of $S^3/Q_{8m}$}\label{appl}

In this section, we apply our results to the computation of $\TC(S^3/Q_{8m})$. We first set a presentation of $Q=Q_{8m}$,
\begin{equation}
    \label{Q8n pre}
    \langle \mathbf{x}, \mathbf{y}\mid \mathbf{x}^{4m} = \mathbf{y}^2,\, \mathbf{y}^4 = 1,\, \mathbf{y}\mathbf{x}\mathbf{y}^{-1}= \mathbf{x}^{-1} \rangle.
\end{equation}
\noindent
For simplicity, we denote $\mathbf{x}^{4m} = \mathbf{y}^2$ by $-1$. Then any element in $Q$ is of the form $\pm \mathbf{x}^{i} \mathbf{y}^j$ for $0\le i< 4m, j = 0, 1$. The centralizer of each element is given as follows.
\begin{equation}
    \label{cent}
    C_{Q}(\pm 1) = Q,\, C_{Q}({\pm \mathbf{x}^{i}}) = \langle \mathbf{x}\rangle \cong \Z/4m,\, C_{Q}({\pm \mathbf{x}^{i}\mathbf{y}}) = \langle \mathbf{x}^{i}\mathbf{y}\rangle \cong \Z/4.
\end{equation}
We denote by $\varpi$ the quotient
\[
Q \twoheadrightarrow V = \langle \tilde{\mathbf{x}}, \tilde{\mathbf{y}}\mid \tilde{\mathbf{x}}^{2} = \tilde{\mathbf{y}}^2= 1,\, \tilde{\mathbf{y}}\tilde{\mathbf{x}}\tilde{\mathbf{y}}^{-1}= \tilde{\mathbf{x}}^{-1} \rangle \cong (\Z/2)^2,
\]
where we set $\tilde{a} = \pi(a)\in V$ for $a\in Q$.

\begin{lemma}
    \label{pi triv}
    Let $a, b\in Q\setminus\mathrm{Ker}\,\varpi$. If $a\notin C_{Q}(b)$, then the homomorphism
    \[
    \Psi_{(a,b),(\tilde{a},\tilde{b})}\colon H^r(C_{V}({\tilde{a},\tilde{b}}); A)\to H^r(C_{Q}({a,b}); A)
    \]
    is trivial for any $\Z$-module $A$.
\end{lemma}

\begin{proof}
    By \eqref{cent}, we have $C_{Q}({a,b}) = \langle \pm 1\rangle$ if and only if $a\notin C_{Q}(b)$. Then since $\langle \pm 1\rangle \subset \mathrm{Ker}\,\varpi$, we obtain the statement by Theorem \ref{decomp}.
\end{proof}

Hereafter in this section, we only consider the mod $2$ cohomology unless otherwise stated, and so we set $A=\Z/2$. We next recall the mod $2$ cohomology of $Q_{8m}$. 
Let $t\in H^{1}(\Z/2)$ denote the generator, and let $x, y\in H^1(V)$ denote the images of $t$ by the projections $V\twoheadrightarrow \langle \tilde{\mathbf{x}} \rangle$ and $V\twoheadrightarrow \langle \tilde{\mathbf{y}} \rangle$, respectively. Then $x$ and $y$ generates $H^*(V)$, and $\varpi$ yields a homomorphism
\begin{align}
    \label{HQ}
    H^*(V) \cong &\,A [x, y]\\ 
    &\notag\xrightarrow{\varpi^*}
    \begin{cases}
    A [x, y, e]/(x^2+xy+y^2, x^2y+xy^2) & (m \text{ is odd})\\
    A [x, y, e]/(xy, x^3+y^3)& (m \text{ is even})
\end{cases}\\
&\cong\notag H^*(Q),
\end{align}
where $x = \varpi^*(x), y = \varpi^*(y)$ and $|e|=4$ (for details, see \cite[IV. Lemmas 2.10 and 2.11]{AM} and \cite[\S 3]{TZ}.)

\begin{lemma}
    \label{V gens}
    For $a\in Q\setminus\langle\pm 1\rangle$, the composite $C_Q(a) \hookrightarrow Q \xrightarrow{\varpi} V$ induces a homomorphism
    \begin{align*}
    H^*(V)\cong A [x, y] \to A [u, v]/(u^2) \cong H^{\ast}(C_Q(a)),
        x\mapsto iu, y\mapsto ju \,\text{ for }a= \pm \mathbf{x}^{i}\mathbf{y}^{j}.
\end{align*}
\noindent
where $|u|=1$ and $|v|=2$.
\end{lemma}

\begin{proof}
    For an integer $n\ge1$, it is well known that the quotient $\Z/4n \twoheadrightarrow\Z/2$ induces a homomorphism $H^{1}(\Z/2) = \langle t \rangle \to \langle u \rangle = H^{1}(\Z/4n)$ with $t\mapsto u$. Thus by \eqref{cent}, we obtain a homomorphism as stated.
\end{proof}

We now proceed to the proof of Theorem \ref{main}. For $z\in H^*(G)$, let us denote $\tau(z)= z\otimes 1 + 1\otimes z\in H^*(G^2)$, which is a zero-divisor. We consider three cohomology classes in $H^3(V^2)$,
\begin{gather*}
    v_1 := \tau(x)^2(1\otimes y)+\tau(x)\tau(y)(x\otimes 1)\\
    v_2 := \tau(x)^2(1\otimes y)+\tau(x)\tau(y)(1\otimes y)\\
    v_3 := \tau(x)^2(y\otimes 1)+\tau(y)^2(x\otimes 1).
\end{gather*}
We recall the derived couple \eqref{der seq}, denoting its terms by $D^n_{p, q}(G)$ and $E^n_{p, q}(G)$ for distinction. Observe that, by Lemma \ref{Ext exact hom}, the sequence \eqref{der seq} is natural with respect to a group homomorphism $\alpha\colon H\to G$. 

\begin{lemma}
    \label{E3 Z2}
    There are identities $E^{3}_{1, 2}(\Z/2) = E^{1}_{1, 2}(\Z/2) = \Z/2$.
\end{lemma}

\begin{proof}
    Since $(\Z/2 \setminus \{1\})^2$ consists of a single tuple, we obtain the second identity by \eqref{cent iso}. Remark that, by \cite[Theorem 6.1]{FM}, there exists an element $z\in D^{3}_{3,0}(\Z/2)$ that is not contained in $D^{4}_{3,0}(\Z/2)$, implying that $E^{3}_{1, 2}(\Z/2)\neq 0$. Thus the first identity is obtained.
\end{proof}

\begin{lemma}
    \label{wgt 3}
    For $i = 1, 2, 3$, $\varpi^{*}(v_i) \in D^{4}_{3, 0}(Q)$.
\end{lemma}

\begin{proof}
    By Theorem \ref{FM Thm}, Proposition \ref{FM zero} and Lemma \ref{FM supp}, we have that $v_i\in D^3_{3, 0}(V)$ for $i = 1, 2, 3$. Then by the naturality, it follows that $\varpi^{*}(v_i) \in D^{3}_{3, 0}(Q)$. For each $i$, we take a representative $z_i\in E^1_{1, 2}(V)$ to satisfy $\varepsilon^{(3)}(v_i) = [z_i] \in E^3_{1, 2}(V)$.
    
    Since $V$ is abelian, the centralizer $C_V(c)$ equals to $V$ for any $c\in V$. We aim to determine the image $q_{(c, c)}(z_i)\in H^1(V)$ for each $c\in V\setminus\{1\}$ and $i$. As an instance, consider the case $c = \tilde{\mathbf{x}}$ and $i = 1$. Then by Theorem \ref{decomp}, the inclusion $\iota\colon \Z/2 = \langle \tilde{\mathbf{x}} \rangle \hookrightarrow V$ induces a homomorphism
    \[
    \Phi = \Phi_{(\tilde{\mathbf{x}}, \tilde{\mathbf{x}}),(\tilde{\mathbf{x}},\tilde{\mathbf{x}})}\colon H^1(V)\to H^1(\Z/2),\quad x\mapsto t\text{ and }y\mapsto 0.
    \]
    Remark that $\iota(z_1)=0\in D^4_{3,0}(\Z/2)\subset H^3({\Z/2})$. Then by Lemma \ref{E3 Z2}, we obtain $\Phi^* (q_{(\tilde{\mathbf{x}}, \tilde{\mathbf{x}})}(z_1)) = 0$ and so $q_{(\tilde{\mathbf{x}}, \tilde{\mathbf{x}})}(z_1) = ky$, where $k$ is either $0$ or $1$. By similar arguments, we get
    \[
    q_{(\tilde{\mathbf{x}}, \tilde{\mathbf{x}})}(z_i) = k_{i, 1}y,\quad
        q_{(\tilde{\mathbf{y}}, \tilde{\mathbf{y}})}(z_i) = k_{i, 2}x,\quad
        q_{(\tilde{\mathbf{xy}}, \tilde{\mathbf{xy}})}(z_i) = k_{i, 3}x+k_{i, 3}y,
    \]
    where $k_{i, j}$ is either $0$ or $1$ for $1\le i, j\le 3$. Then by Theorem \ref{decomp} and Lemma \ref{V gens},
    \[
    \Phi_{(a, b), (c, c)}(q_{(\tilde{a}, \tilde{a})}(z_i)) = 0 \quad\text{ for any }a, b \in Q\text{ with }\tilde{a} = \tilde{b} = c.
    \]
    By Theorem \ref{decomp} and Lemma \ref{pi triv}, the homomorphism $\Phi_{(a, b), (c, d)}$ is trivial for elements $a, b \in Q$ and $c, d\in V$ with $c\neq d$, $\tilde{a}\neq\tilde{b}$ or $\tilde{a}\neq c$. Then we obtain $\varpi^*(z_i) = 0\in E^1_{1,2}(Q)$, and by the naturality of the derived sequences,
    \[
    \varepsilon^{(3)}(\varpi^*(v_i)) = \varpi^*[z_i] = [\varpi^*(z_i)] = 0 \in E^3_{1, 2}(Q).
    \]
    Thus the statement is proved.
\end{proof}

We are now ready to prove Theorem \ref{main}.

\begin{proof}[Proof of Theorem \ref{main}]
    The canonical map $p\colon X = S^3/Q \to BQ$ induces an isomorphism on mod $2$-cohomology $p^*\colon H^n(Q)\xrightarrow{\cong} H^n(X)$ for $n\le 3$, and obviously $H^n(X)=0$ for $n\ge4$. Then by \eqref{HQ}, we have
    \begin{align*}
        (p^2)^*(\varpi^*(v_1)\smile\varpi^*(v_3))
        &=(p^2)^*(\varpi^*(v_1))(p^2)^*(\varpi^*(v_3))\\
        &=(1\otimes x^2y+x^2y\otimes1+xy\otimes x+x\otimes xy)(y\otimes x^2+x\otimes y^2)\\
        &= x^2y\otimes xy^2 \neq 0 \in H^6(X^2)
    \end{align*}
    \noindent
    for $m$ odd, and
    \begin{align*}
        (p^2)^*(\varpi^*(v_2)\smile\varpi^*(v_3))
        &=(p^2)^*(\varpi^*(v_2))(p^2)^*(\varpi^*(v_3))\\
        &=(x^2\otimes y + x\otimes y^2)(y\otimes x^2+x\otimes y^2)\\
        &= x^3\otimes y^3 \neq 0 \in H^6(X^2)
    \end{align*}
    \noindent
    for $m$ even. Then, for the canonical class $\mathfrak{v}_X\in H^1(X^2; I_Q)$, we obtain $\mathfrak{v}_X^6\neq 0\in H^6(X^2; I_Q^{\otimes 6})$ by Theorem \ref{FM Thm} and Lemmas \ref{FM supp}, \ref{wgt 3}, which implies
    \[
    \TC(S^3/Q)\ge \wgt(\mathfrak{v}_X)\ge 6.
    \]
    Thus since $\TC(S^3/Q)\le 2\dim (S^3/Q) = 6$, the proof is completed.
\end{proof}

\section{Final remarks}\label{rem}

We conclude this paper with some remarks on our results, and some open problems regarding the topological complexity of non-simply connected spaces. We hope that these problems precipitate further research.

First, we admit that we have not provided an explicit description of the differentials of $E^n$. While the statement of Lemma \ref{E3 Z2} implies triviality of the differentials, it is possible only because the group $G = \Z/2$ and the set of tuples $(G\setminus\langle1\rangle)^s$ are enough simple. B\l aszczyk et al. \cite{BCE} mentioned the first differential $d^1$ in more general setting, yet their description is not practical for computation. Therefore, it is essential to investigate the differentials and $E^n$-terms for $n\ge 2$. We believe that it concerns with Farrell cohomology theory, as a spectral sequence in \cite[X. (4.1)]{Br} has some resemblance to \eqref{cent iso}.

Next, we remark on the determination of $\TC(S^n/G)$ for each free action $G\curvearrowright S^n$. Finite groups that act freely on spheres are completely characterized by Madsen et al. \cite{MTW}. In particular, Milnor \cite{M} classified a finite group $G$ acting freely on $S^3$.
\begin{question}
    For each finite group $G$ acting freely on $S^3$, determine $\TC(S^3/G)$. More generally, determine $\TC(S^n/G)$ for each finite group $G$ with free action on sphere.
\end{question}
\noindent
The first question appears to be difficult to address, since the Poincar\'{e} sphere $S^3/(2I)$ has the same integral cohomology as $S^3$. The second quection is more intractable; even the topological complexity of the lens space $L^{2n+1}_p = S^{2n+1}/(\Z/p)$ is not completely determined. We expect that our method contributes to further studies on this problem.

Finally, we mention further generalizations to address various types of motion planning problem. There have been many analogs of topological complexity considered. The most significant one is the sequential (or higher) topological complexity introduced by Rudyak \cite{R}, which arose from the robot motion planning passing through several intermediate points; for a space $X$ and an integer $r\ge 2$, the invariant is restated as $\TC_r(X) : = \secat(\Delta\colon X\to X^r)$, where $\Delta$ denotes the diagonal. Farber and Oprea \cite{FO} asked:
\begin{question}
    \label{growth}
    For which space $X$, is the formal series $\sum_{r=1}^{\infty}\TC_{r+1}(X)z^{r}$ a rational function $P(z)/{(1-z)^2}$ such that $P(z)$ is a polynomial with $P(1)= \cat(X)$?
\end{question}
\noindent
Notice that there exists a family of spaces not satisfying the above condition \cite{FKS}. As a recent progress, \cite{ACGHV} collected affirmative cases of spaces with finite cyclic fundamental group. We now state a conjecture:
\begin{conjecture}
    For $m\ge 1$, the $r$-th sequential topological complexity of $S^3/Q_{8m}$ is equal to $3r$, supporting Question \ref{growth}. 
\end{conjecture}
\noindent
We also anticipate that our method will be applied to other variants, for instance, parametrized topological complexity introduced by Cohen et al. \cite{CFW}.

\end{document}